\documentclass[oneside, a4paper, 11pt]{amsart}
\usepackage[utf8]{inputenc}
\usepackage[T1]{fontenc}
\usepackage{amsmath,amsthm,amsrefs,amssymb,bbm,comment,mathrsfs}

\calclayout

\pdfoutput=1
\usepackage[ttscale=.85]{libertine}
\usepackage{libertinust1math}

\usepackage[tracking=true]{microtype}
\SetTracking
  {encoding=T1, shape=sc}
  {10}
\SetProtrusion
  {encoding=T1,size={7,8}}
  {1={ ,750},2={ ,500},3={ ,500},4={ ,500},5={ ,500},
    6={ ,500},7={ ,600},8={ ,500},9={ ,500},0={ ,500}}

\usepackage{microtype}

\usepackage{graphicx,enumitem,stmaryrd}

\usepackage{anyfontsize}
\allowdisplaybreaks[4]

\usepackage[table]{xcolor}

\usepackage{adjustbox}

\newcommand{\K}{\mathbb{K}}
\newcommand{\Z}{\mathbb{Z}}

\newcommand{\R}{\mathbb{R}}
\newcommand{\id}{\mathrm{id}}
\newcommand{\Ker}{\mathrm{Ker}}
\newcommand{\ov}[1]{\overline{#1}}

\newtheorem{theorem}{Theorem} 
\newtheorem{lemma}[theorem]{Lemma} 
\newtheorem{proposition}[theorem]{Proposition}

\newtheorem{exa}[theorem]{Example}
\newtheorem{remark}[theorem]{Remark}

\title{Generalized Kac-Paljutkin  algebras}
\author{Christian Lomp}
\dedicatory{This paper is dedicated to my father.}
\address{Department of Mathematics, University of Porto, 4169-007 Porto, Portugal}
\email{{clomp@fc.up.pt}}
\begin{document}

\begin{abstract}
In this note, we construct a family of non-commutative, non-cocommutative  semisimple Hopf algebras \( H_{n,m} \) of dimension \( n^m m! \) over a field of characteristic zero containing a primitive \( n \)th root of unity, where \( n, m \geq 2 \) are integers. The well-known eight-dimensional Kac--Paljutkin algebra arises as the special case \( H_{2,2} \), while the Hopf algebras previously constructed by Pansera correspond to the instances \( H_{n,2} \). Each algebra \( H_{n,m} \) is defined as an extension of the group algebra \( \mathbb{K} \Sigma_m \) of the symmetric group by the \( m \)-fold tensor product \( R = \mathbb{K} \mathbb{Z}_n^{\otimes m} \), where \( \mathbb{Z}_n \) denotes the cyclic group of order \( n \). This extension admits a realization as a crossed product:
\(H_{n,m} = \mathbb{K} \mathbb{Z}_n^{\otimes m} \#_\gamma \Sigma_m.\)
In the final section, we construct a family of irreducible \( m \)-dimensional representations of \( H_{n,m} \) that are inner faithful as \( R \)-modules and exhibit a nontrivial inner-faithful action of a subalgebra of \( H_{n,m} \) on a quantum polynomial algebra.

\end{abstract}

\maketitle

\section{Introduction}
Let $\K$ be a field of characteristic zero. 
In \cite{KacPaljutkin} (see also \cite{Masuoka}), Kac and Paljutkin introduced an eight dimensional semisimple Hopf algebra $H_8$ given by generators $x,y,z$ subject to the relations 
\begin{equation}\label{eq:1} zx=yz, \quad zy=xz, \quad xy=yx, \quad x^2=y^2=1, \quad z^2=\frac{1}{2}\left(1+x+y-xy\right),\end{equation}
such that $x$ and $y$ are group like and 
\begin{equation}
\Delta(z)=\frac{1}{2}\left(1\otimes 1 + x\otimes 1 + 1\otimes y - x\otimes y\right)(z\otimes z).\end{equation}
The antipode of $H_8$ is the identity map. The algebra $H_8$ can be seen as a quotient of a skew polynomial ring $R[z;\sigma]/I$, where the ring of coefficients is the group ring $R=\K[\Z_2\times \Z_2]$, with an automorphism $\sigma$ that interchanges the generators $x$ and $y$ of $\Z_2\times \Z_2$ and $I$ is the ideal generated by the relation of $z^2$ in equation (\ref{eq:1}). 
In \cite{PanseraPhD, Pansera}, Pansera introduced  a $2n^2$-dimensional semisimple Hopf algebra $H_{2n^2}$ as $R[z;\sigma]/I$, where $R=\K[\Z_n\times \Z_n]$, for a given integer $n\geq 2$. The comultiplication of $z$ is given by 
\begin{equation}
\Delta(z) = \frac{1}{n} \left(\sum_{i,j=0}^{n-1} q^{-ij} x^i\otimes y^j\right)(z\otimes z),
\end{equation}
where $q$ is a primitive $n$th root of unity in $\K$. The group algebra $R$ is a Hopf subalgebra of $H_{2n^2}$ and the antipode $S$ of $H_{2n^2}$ is extended from $R$ by setting $S(z)=z$. Motivated by Etingof-Walton's Theorem, saying that if a semisimple Hopf algebra acts on a commutative domain then its action factors through a group algebra, Pansera showed in \cite{Pansera}, that $H_{2n^2}$ admits an action on a quantum plane $A$ that does not factor through a group algebra, i.e. $H_{2n^2}$ acts inner-faithfully on $A$.

The purpose of this note is to extend Pansera's construction and to provide further examples of non-trivial semisimple  Hopf algebras acting inner-faithfully on non-commutative domains. The basic observation is that, in the construction of $H_{2n^2}$, the action of the automorphism $\sigma$ of order $2$, can be regarded as the action of the symmetric group $\Sigma_2$ on the tensor product $\K[\Z_n] \otimes \K[\Z_n] = \K[\Z_n\times \Z_n]$. As a generalization we consider $\Sigma_m$ acting on the $m$-fold tensor $R=B^{\otimes m}$, for some bialgebra $B$.\footnote{We will denote the symmetric group by $\Sigma_m$ instead of $S_m$ in order to avoid confusion with the antipode $S$.} The standard generators $s_1, \ldots, s_{m-1}$ of $\Sigma_m$, i.e. the transpositions $s_i=(i, i+1)$, generate a free monoid $M=\langle \ov{s}_1, \ldots, \ov{s}_{m-1}\rangle$ that acts on $R$ and allows to consider the skew monoid algebra $R\# M$, which in the case $m=2$ corresponds to the skew polynomial ring $R[z;\sigma]$. The comultiplication of $\ov{s_i}$ can be defined as $\Delta(\ov{s}_i) = J_i (\ov{s}_i \otimes \ov{s}_i )$, for a suitable twist $J_i\in R\otimes R$. Under further assumptions on these twists $J_i$, we define a Hopf structure on the quotient $R\# M / I$, where $I$ is the ideal generated by the usual relations of the symmetric group, $\ov{s}_i\ov{s}_{i+1}\ov{s}_i = \ov{s}_{i+1}\ov{s}_i\ov{s}_{i+1}$ and $\ov{s}_i\ov{s}_j=\ov{s}_j\ov{s}_i$ for $|i-j|>1$, but with $\ov{s}_i^2 = t_i := \mu_R(J_i)$, where $\mu_R$ is the multiplication of $R$. The obtained Hopf algebra $H$ is an extension of  $\K\Sigma_m$ by $R$ and can be  considered a crossed product $R\#_{\gamma} \Sigma_m$ for a suitable $2$-cocycle $\gamma:\Sigma_m\times \Sigma_m \to R^\times$. For $B=\K[\Z_n]$ and $\K$ containing a primitive $n$th root of unity, we provide twists $J_i$ that satisfy all our assumptions and yield a family of semisimple Hopf algebras $H_{n,m}=\K[\Z_n]^{\otimes m}\#_{\gamma} \Sigma_m$ of dimension $n^m m!$. The original Kac-Paljutkin Hopf algebra appears as $H_{2,2}$, while Pansera's algebras appear as $H_{n,2}$.

In the last section we examine actions of $H_{n,m}$ on a quantum polynomial algebra $A$ in $m$ generators $u_1, \ldots, u_m$. We will show that there exists a  $mn^m$-dimensional semisimple Hopf subalgebra $H$ of $H_{n,m}$  that acts inner-faithfully on  $A$ and whose ring of invariants $A^H$ is the subring of cyclic polynomials in $u_1^n, \ldots, u_m^n$, in case $n$ is even.

\section{Constructing Hopf algebras via twists}
An invertible element $J\in B\otimes B$ of a bialgebra $B$ is called a \emph{twist for $B$} if 
\begin{equation}\label{eq:twist}
(\Delta_B \otimes \id_B)(J)(J\otimes 1_B) = (\id_B \otimes \Delta_B)(J)(1_B\otimes J)
\end{equation}
and
$(\epsilon \otimes \id )(J)=1=(\id\otimes\epsilon)(J)$ hold.  
Writing symbolically $J=J^{(1)} \otimes J^{(2)}$, equation (\ref{eq:twist}) means
\begin{equation}\label{eq:twist2}
J_1^{(1)} j^{(1)} \otimes J_2^{(2)} j^{(2)} \otimes J^{(2)} 
= 
J^{(1)} \otimes J_1^{(2)} j^{(1)} \otimes J_2^{(2)} j^{(2)},
\end{equation}
where $j^{(1)} \otimes j^{(2)}$ is another copy of $J$ and where $\Delta_B(J^{(i)})=J_1^{(i)}\otimes J_2^{(i)}$ denotes the comultiplication of one of the legs of $J^{(i)}$ of $J$.  Clearly $1_B\otimes 1_B$ is a twist for $B$. For $B=\K\Z_n$, with generator $x$, the element $J=\frac{1}{n} \sum_{i,j=0}^{n-1} q^{-ij} x^i \otimes x^j$ is a twist for $B$. We will use twists  to deform the comultiplication of a group-like element of a bialgebra as they appear naturally in our construction. Suppose $R\subseteq S$ is an extension of bialgebras with $x\in S\setminus R$ an element such that $Rx$ is a free $R$-module. Suppose further that $J\in R\otimes R$ is an invertible element, such that $\Delta_S(x)=J(x\otimes x)$ and $\epsilon_S(x)=1$. Then the coassociativity for $\Delta_S(x)$ and the freeness of $Rx$ imply  equation (\ref{eq:twist}), while the counity will imply $(\epsilon_R \otimes \id_R)(J)=(\id_R \otimes \epsilon_R)(J)=1$. Hence $J$ is a twist for $R$. Since our aim is to extend the bialgebra structure of some algebra $R=B^{\otimes m}$ to a skew monoid algebra $R\# M$, where $M$ is generated by some elements $\ov{s}_1, \ldots, \ov{s}_k$, such that $\Delta(\ov{s}_i) = J_i (\ov{s}_i \otimes \ov{s}_i)$ holds for some $J_i\in R\otimes R$,  we will need that $J_i$ is a twist for $R$.

For two  bialgebra-twist pairs  $(B,J)$ and $(B',J')$, the element $(\id_B \otimes \tau_{B,B'} \otimes \id_{B'})(J\otimes J')$ is a twist for $R=B\otimes B'$, where $\tau_{B,B'}:B\otimes B' \to B'\otimes B$ denotes the usual flip map.
Hence, given a twist $J=J^{(1)}\otimes J^{(2)}$ for $B$, we could consider 
$$J'  = \left(J^{(1)}\otimes 1_B \right)\otimes \left(J^{(2)} \otimes 1_B\right) = (\id_B \otimes \tau_{B,B} \otimes \id_B)(J \otimes 1_{B^{\otimes 2}}),$$ which is a twist for $B\otimes B$. However, we would like to embed $J$ differently into $(B\otimes B) \otimes (B\otimes B)$. For this, let $m\geq 2$ and denote by $e_i^m: B \to B^{\otimes m}$, the embedding of  $B$ into the $i$th tensorand of $B^{\otimes m}$. We would like that $(e_i^m\otimes e_j^m)(J) \in B^{\otimes m} \otimes B^{\otimes m}$ is a twist for $B^{\otimes m}$, for any $1\leq i<j\leq m$. For $m=2$, we need in particular that 
$ (e_1^2\otimes e_2^2)(J)$ is a twist for $B\otimes B$. For $J=J^{(1)} \otimes J^{(2)}$, this condition is equivalent to the condition that 
 $$J'  = (e_1^2\otimes e_2^2)(J) = \left(J^{(1)}\otimes 1_B\right) \otimes \left(1_B\otimes J^{(2)}\right)$$  satisfies equation (\ref{eq:twist}) for the ring $B\otimes B$, which boils down to
\begin{equation}\label{eq:strongtwist}
J_1^{(1)}j^{(1)}  \otimes J_2^{(1)} \otimes j^{(2)} \otimes J^{(2)}
=
J^{(1)}   \otimes j^{(1)}\otimes J_1^{(2)} \otimes J_2^{(2)}j^{(2)},
\end{equation}
where $j^{(1)} \otimes j^{(2)}$ is another copy of $J$. Conversely, equation (\ref{eq:strongtwist})  is actually equivalent to
$ (e_1^2\otimes e_2^2)(J)$ being a twist for $B\otimes B$, because the counit condition is  automatically verified. Comparing  (\ref{eq:strongtwist}) and (\ref{eq:twist2}) shows that the difference lies in the middle tensorands. If  $J$ is central in $B\otimes B$ then equation (\ref{eq:strongtwist}) implies (\ref{eq:twist2}). It remains unclear whether the converse holds.

\begin{lemma}\label{lem:twist} Let $J$ be a  twist for $B$. Suppose
\begin{equation}\label{eq:strong_twist} (e_1^2\otimes e_2^2)(J) \mbox{   is a twist for } B \otimes B.\end{equation} Then  $(e_i^m\otimes e_{j}^m)(J)$ is a twist for $B^{\otimes m}$, for any $1\leq i<j\leq m$. 
\end{lemma}
\begin{proof} Set $R=B^{\otimes m}$.
To ease notation, we will drop the superscript $m$ from $e_i^m$. 
Let  $J'=(e_i\otimes e_{j})(J) \otimes R\otimes R$. Note that $\Delta_R e_i = (e_i\otimes e_i) \Delta_B$ and $1_R = e_i(1_B)$. Then
\begin{eqnarray*}
( \Delta_R \otimes id_R)(J')(J' \otimes 1_R)
&=&
 ( \Delta_R e_i \otimes e_j) (J) (e_i\otimes e_j \otimes e_j)(J\otimes 1_B)\\
&=&
(e_i\otimes e_i \otimes e_j) ( \Delta_B\otimes \id_B)(J) (e_i\otimes e_j \otimes e_j )(J\otimes 1_B)\\
&=&
e_i(J^{(1)}_1j^{(1)}) \otimes e_i(J^{(1)}_2)e_j(j^{(2)}) \otimes e_j(J^{(2)})\\
&=&
(e_i\otimes e_{i,j}\otimes e_j)\left(J^{(1)}_1j^{(1)} \otimes J^{(1)}_2 \otimes j^{(2)} \otimes J^{(2)}\right),
\end{eqnarray*}
where $j^{(1)}\otimes j^{(2)}$ is another copy of $J$ and 
$e_{i,j}: B^{\otimes 2} \to R$ denotes the embedding of $B\otimes B$ into the $i$th and $j$th tensorands of $R$.
Similarly one shows $( \id_R \otimes \Delta_R )(J')(1_R\otimes J' ) = (e_i\otimes e_{i,j}\otimes e_j)\left(J^{(1)}   \otimes j^{(1)}\otimes J_1^{(2)} \otimes J_2^{(2)}j^{(2)}\right).$
Since $J$ satisfies equation (\ref{eq:strongtwist}), $J^{(1)}_1j^{(1)} \otimes J^{(1)}_2 \otimes j^{(2)} \otimes J^{(2)} = 
 j^{(1)}\otimes J_1^{(2)} \otimes J_2^{(2)}j^{(2)}$ and therefore 
$( \Delta_R \otimes id_R)(J')(J' \otimes 1_R) = ( \id_R \otimes \Delta_R )(J')(1_R\otimes J' ).$
\end{proof}

The main idea of this note is to use the action of the symmetric group $\Sigma_m$ on the tensor product $R=B^{\otimes m}$ and to extend the bialgebra structure of $R$ to that of a suitable quotient of a skew monoid algebra. The comultiplication of the newly added generators however shall be twisted by some twists of $R$. A pair $(\sigma, J)$ of an algebra homomorphism $\sigma:R\to R$ and a twist $J$ is called a twisted homomorphism (see \cite{Davydov}) if $\epsilon \sigma = \epsilon$ and $\Delta \sigma (h) = J(\sigma \otimes \sigma)\Delta(h)J^{-1}$, for $h\in R$. We will be mostly  interested in the case where $J$ commutes with the elements of the image of $\Delta$, in which case  $(\sigma, J)$ is a twisted homomorphism if and only if $\sigma$ is a coalgebra homomorphism. Pansera showed in \cite{Pansera}, that if  $R$ is a bialgebra with twisted homomorphism  $(\sigma ,J)$, then  the 
bialgebra structure of $R$ can be extended to the skew polynomial ring $R[x;\sigma]$, such that $\Delta(x)=J(x\otimes x)$ and $\epsilon(x)=1$.  We will extend this result to $R= B^{\otimes m}$ and a twist $J$ for $B$, that satisfies  (\ref{eq:strong_twist} ). Here the skew polynomial ring will be replaced by a skew monoid algebra. 

Each permutation $s \in \Sigma_m$ acts on $R=B^{\otimes m}$ by $\sigma_{s}: R\to R$, with $\sigma_{s} (a_1\otimes \cdots \otimes a_n) = a_{s(1)} \otimes \cdots \otimes a_{s(n)}$.  Note that $\sigma_{s}$ is a bialgebra homomorphism of $R$. Thus 
$\Delta \sigma_{s}  = (\sigma_{s} \otimes \sigma_{s})\Delta$ and $\epsilon\sigma_{s} = \epsilon$. 
Given a twist $J$ for $B$, such that $(e_1^2\otimes e_2^2)(J)$   is a twist for $B \otimes B$,
 we define for each transposition $s = (i j)\in \Sigma_m$, with $i<j$, the element $J_s = (e_i^m \otimes e_j^m)(J)\in R\otimes R$, which is a twist for $R$ by Lemma \ref{lem:twist}.

\begin{theorem}\label{thm:bialgebra}
Let $B$ be a bialgebra with central  twist $J$ of $B$ that satisfies (\ref{eq:strong_twist}).
Let $m\geq 2$ and $X\subseteq \Sigma_m$  a set of transpositions. Then the free monoid $M$, generated by $\{\ov{s} : s\in X\}$,  acts on  $R=B^{\otimes m}$ and the bialgebra structure on $R$ extends to a bialgebra structure on the skew monoid algebra $R\# M$, such that $\Delta(\ov{s}) = J_{s} (\ov{s} \otimes \ov{s})$ and $\epsilon(\ov{s})=1$, for all $s\in X$. Moreover, $\pi:R\# M \to \K[M]$ with $\pi(a\# \overline{w}) =  \epsilon(a)\overline{w}$ is a homomorphism of bialgebras.
\end{theorem}

\begin{proof} Denote by $\mathrm{Aut}(R)$ the group of algebra automorphisms of $R$.
 Since $M$ is the free monoid on $X$, the map $X\to \mathrm{Aut}(R)$ given by $s\mapsto \sigma_s$ extends to a map of monoids $M\to \mathrm{Aut}(R)$, which allows to consider the skew monoid ring $H=R\# M$, whose basis as a left $R$-module is given by the words in $\{\ov{s}: s\in X\}$. We identify  elements $a\in R$  with elements $a\#1_M$, where $1_M$ is the identity of $M$. In particular, for each $s\in X$ and $a\in R$, we have $ \ov{s}a = \sigma_s(a)\ov{s}.$  To define $\Delta_H: H\to H\otimes H$, we set 
\begin{equation}
\Delta_H(a\ov{w}) := \Delta_R(a) J_{s_1}(\ov{s}_1\otimes \ov{s}_1) \cdots J_{s_k}(\ov{s}_k\otimes \ov{s}_k),
\end{equation}
for $a\in R$ and $w=s_1\ldots s_k$ a word in $X$,
 In particular,  ${\Delta_H}_{\mid_R}=\Delta_R$ and  $\Delta_H(\ov{s}) = J_s (\ov{s}\otimes \ov{s})$, for  $s\in X$.
Since $J$ is central in $B^{\otimes 2}$,  $J_s$ is also central in $R\otimes R$, and since $\sigma_s$ is a bialgebra homomorphism we have
\begin{equation} \Delta_H(\ov{s})\Delta_R(a)
= J_s(\ov{s}\otimes \ov{s})\Delta_R(a) 
= J_s (\sigma_s\otimes \sigma_s)\Delta_R(a) (\ov{s} \otimes \ov{s})
= \Delta_R(\sigma_s(a)) \Delta_H(\ov{s}).\end{equation}
Thus, $\Delta_H$ is a well-defined algebra homomorphism.
The coassociativity follows from the fact that $J_s$ is a twist for $R$, by Lemma \ref{lem:twist}:
\begin{eqnarray*}
(\Delta_H\otimes \id_H)\Delta_H(\ov{s}) 
&=& (\Delta_R\otimes \id_R)\Delta_R(J_s) \left(\Delta_H(\ov{s})\otimes \ov{s}\right)\\
&=& (\Delta_R\otimes \id_R)\Delta_R(J_s) (J_s\otimes 1_R) (\ov{s} \otimes \ov{s}\otimes \ov{s})\\
&=& (\id_R\otimes \Delta_R)\Delta_R(J_s) (1_R\otimes J_s) (\ov{s}\otimes \ov{s})\otimes \ov{s})\\
&=& (\id_R \otimes \Delta_R)\Delta_R(J_s) \left(\ov{s}\otimes \Delta_H(\ov{s})\right)\\
&=&(\id_H\otimes \Delta_H)\Delta_H(\ov{s}) 
\end{eqnarray*}
 Furthermore, set $\epsilon_H(a\ov{w})=\epsilon_R(a)$, for all $a\in R$ and $w\in M$. Since $\epsilon_R\sigma_s=\epsilon_R$, we conclude that $\epsilon_H$ is an algebra homomorphism. For any $a\in R$ and $s\in X$:
$$(\epsilon_H\otimes \id_H)\Delta_H(a\ov{s}) 
=  (\epsilon_H\otimes \id_H) \left(\Delta_R(a) J_s (\ov{s}\otimes \ov{s})\right)
=  a (\epsilon_R\otimes \id_R)(J_s)  \ov{s} = a\ov{s},$$
since $(\epsilon_R\otimes \id_R)(J_s) = (\epsilon_B\otimes id_B)(J)=1$. Similarly, one checks 
$(id_H\otimes \epsilon_H)\Delta_H(a\ov{s}) =1$.
\end{proof}

\begin{remark} Note that $R\# \overline{1} =\{ \gamma \in R\#M \mid (id\otimes \pi)\Delta(\gamma) = \gamma \otimes 1\}.$ Hence $R\#M$ is a bialgebra  extension of $\K[M]$ by $R$.

\end{remark}

\begin{remark}
Let $s\in \Sigma_m$ be an element of order $2$. Then $s$ is a product of disjoint transpositions, say $s=s_1\cdots s_k$. 
Since the elements $ J_{s_1}, \ldots, J_{s_k}$ commute pairwise, the product $J_s := J_{s_1}\cdots J_{s_k}$ is also a twist for $R$ and one could prove also a version of Theorem \ref{thm:bialgebra} for subsets $X\subseteq \Sigma_m$ of permutations of order $2$.
\end{remark}
Denote by $\mu_R$ the multiplication of $R$.

\begin{theorem}\label{thm:HopfStructure}
Let $B$ be a bialgebra with $J$ a central  twist for $B$ that satisfies (\ref{eq:strong_twist}).
Let  $m\geq 2$, $X\subseteq \Sigma_m$  a set of transpositions,  $M$ the free monoid generated by $\{\ov{s} : s\in X\}$ acting on $R=B^{\otimes m}$.
\begin{enumerate}
\item Suppose $J$ satisfies
\begin{equation}\label{eq:superstrong}
\Delta_{B\otimes B}(J) = (e_1^2 \otimes e_2^2)(J)(e_2^2 \otimes e_1^2)(J)(J\otimes J).
\end{equation}
Then  the ideal $I=\langle \ov{s}^2 - t_s \: : \: s\in X\rangle$ is a biideal of $R\# M$, where $t_s = \mu_{R}(J_s)$.
\item  Suppose moreover that $B$ is a Hopf algebra with antipode $S$, such that $(S\otimes S)(J)=J$ and $(S\otimes \id)(J) = J^{-1}= (\id\otimes S)(J).$ Then 
$H=(R\#M)/I$ is a Hopf algebra with $S(\ov{s}+I) = \ov{s}+I$ in $H$, for $s\in X$.
\end{enumerate}
\end{theorem}

\begin{proof}
(1)
Let $s = (i j) \in X$, with $i<j$. 
Note that $e_{ij}^m:=\mu_R (e_i^m\otimes e_j^m):B\otimes B\to R$ is an algebra homomorphism that satisfies 
\begin{equation}\label{eq:eij}
\Delta_R\circ e_{ij}^m = (e_{ij}^m\otimes e_{ij}^m)\circ \Delta_{B\otimes B},
\end{equation} as well as $e_{ij}^me _1^2 = e_i^m$ and $e_{ij}^m e_2^2 = e_j^m$. Set $J_s := (e_i^m \otimes e_j^m)(J)\in R^{\otimes 2}$ and 
$t_s = \mu_R(J_s)=e_{ij}^m (J)\in R$.  Since $J$ is central in $B\otimes B$, so is $t_s$ in $R$. Then equations (\ref{eq:superstrong}) and (\ref{eq:eij}) imply
\begin{eqnarray*} 
	\Delta_{R}(t_s) = \Delta_R(e_{ij}^m(J))
	&=& e_{ij}^m\otimes e_{ij}^m \left(\Delta_{B\otimes B}(J)\right) \\
	&=& e_{ij}^m\otimes e_{ij}^m \left((e_1^2\otimes e_2^2)(J)  (e_2^2\otimes e_1^2)(J)(J\otimes J)\right) \\	
	&=& 
	(e_i^m\otimes e_j^m)(J) (e_j^m\otimes e_i^m)(J) (e_{ij}^m(J) \otimes e_{ij}^m(J)) \\
	&=& J_s (\sigma_s\otimes \sigma_s)(J_s)(t_s\otimes t_s),
\end{eqnarray*}
Hence, $\Delta(\ov{s}^2 - t_s) 
	= J_s(\ov{s}\otimes \ov{s})J_s(\ov{s}\otimes \ov{s}) - \Delta(t_s)
	= J_s(\sigma_s\otimes \sigma_s)(J)\left(\ov{s}^2\otimes \ov{s}^2 - t_s\otimes t_s\right)
	$ is an element of $I\otimes (R\# M) + (R\# M)\otimes I$.		
Since $\epsilon_H(t_s)=\epsilon_R(\mu_R(J)))=1=\epsilon_H(\ov{s}^2)$, we conclude that $I$ is a coideal of $R\# M$.

(2) Set $H=(R\# M)/I$ and assume that $B$ is a Hopf algebra with antipode $S$. Then $R$ is  a Hopf algebra with antipode $S^{\otimes m}$, which we will also denote by $S$. Since  $(S\otimes S)(J)=J$ holds, we  have 
$S(t_s) = S(e_{ij}^m(J))=e_{ij}^m(S\otimes S)(J) = t_s.$
Hence, $S(\ov{s})=\ov{s}$ modulo $I$ is well defined. We also set 
\begin{equation}\label{def:S}
S(a\ov{s}_1\cdots \ov{s}_k +I)=\ov{s}_k\cdots \ov{s}_1 S(a)+I,
\end{equation}
 for all $a\in R$ and $s_1,\ldots, s_k \in X$.  Since $S\sigma_s = \sigma_s S$, definition (\ref{def:S}) leads to a well defined algebra anti-homomorphism $S:H\to H$. Writing $J$ again symbolically as $J=J^{(1)} \otimes J^{(2)}$, we calculate modulo $I$:
\begin{eqnarray*} \mu_H(\id_H \otimes S)\Delta_H (\ov{s}) 
	&=& \mu_H(\id_H \otimes S) J_s (\ov{s}\otimes \ov{s})\\
	&=& e_i^m(J^{(1)})\ov{s} S(\ov{s}) S(e_j^m(J^{(2)}))\\
	&=& e_i^m(J^{(1)})S(e_j^m(J^{(2)}))t_s\\
	&=& e_{ij}^m\left( (id\otimes S)(J)\right) e_{ij}^m(J)
	= e_{ij}^m\left( J^{-1}J\right) = 1
\end{eqnarray*}
Similarly $\mu_H(S \otimes \id_H)\Delta_H (\ov{s}) = 1$ holds modulo $I$, because of $(S\otimes \id)(J) = J^{-1}$.
Thus $S$ defines an antipode for the bialgebra $H$.
\end{proof}

Since our aim is to construct semisimple Hopf algebra quotients of $R\# M$, they must be necessarily finite dimensional over $\K$. Thus we will consider quotients of $R\# M$, by using the usual Coxeter  presentation of $\Sigma_m$ with generators $s_i = (i, i+1)$.

\begin{theorem}[Extension by the symmetric group]\label{thm:HopfAlgebra}  Let  $B$ be a Hopf algebra with antipode $S$  and $J$ a central twist for $B$ that satisfies (\ref{eq:strong_twist}) and (\ref{eq:superstrong}), such that  $(S\otimes S)(J)=J$ and $(\id\otimes S)(J)=J^{-1}=(S\otimes \id)(J)$. 
For any $m\geq 2$, let  $X = \{ s_1, \ldots, s_{m-1}\}$ be the set of  transposition $s_i = (i,  i+1)$ that generates $\Sigma_m$ and let $M$ be the free monoid generated by $\{\ov{s_i} : 1\leq i <m \}$, which acts naturally on  $R = B^{\otimes m}$. Consider the  ideal $I$ of $R\# M$ generated by 
$$ \ov{s}_i^2 - t_{s_i}, \qquad \ov{s}_i\ov{s}_{i+1}\ov{s}_i - \ov{s}_{i+1}\ov{s}_i\ov{s}_{i+1},  \qquad \ov{s}_i\ov{s}_{j} - \ov{s}_j\ov{s}_i,\qquad\forall i,j \mbox{ with } |i-j|>1.$$
Then $H=(R\# M)/I$ is a Hopf algebra of dimension $\dim(B)^{m} m!$ with Hopf subalgebra $R=B^{\otimes m}$ and Hopf quotient $\K \Sigma_m$. Furthermore, 
\begin{enumerate}
\item $\{\ov{w}: w\in \Sigma_m\}$ is an $R$-basis of $H$;
\item there exists a $2$-cocycle $\gamma:\Sigma_m\times \Sigma_m \to Z(R)^\times$ such that $H \simeq R \#_\gamma \Sigma_m$ with multiplication given by 
\begin{equation}
(a\# \ov{w})(b\# \ov{v}) = 
a\sigma_w(b)  \gamma(w,v)\# \ov{wv}.
\end{equation}
for all $a,b \in R$ and $w,v\in \Sigma_m$, where $\gamma(w,v)$ is a product of elements $\sigma_u(t_s)$, for  some $u\in \Sigma_m$ and $s\in X$.
\item $\epsilon(\gamma(w,v))=1$, for all $w,v\in \Sigma_m$.
\item if $B$ is semisimple, then $H$ is semisimple with integral $\int \sum_{w\in \Sigma_m} \ov{w}$, where $\int$ is the integral of $R$.
\end{enumerate}
\end{theorem}

\begin{proof}	
From Theorem \ref{thm:HopfStructure} we know already, that $\langle \ov{s}_i^2-t_{s_i}: 1\leq i < m\rangle$ is a coideal of $R\# M$.
Furthermore, for $i,j$ such that $|i-j|>1$, the transpositions $s_i$ and $s_j$ are disjoint. Hence $J_{s_i}$ and $\ov{s}_j\otimes \ov{s}_j$, as well as $J_{s_j}$ and $\ov{s}_i\otimes \ov{s}_i$ commute in $R\# M$. Thus
$$\Delta(\ov{s}_i\ov{s}_j - \ov{s}_j\ov{s}_i)
= J_{s_i}J_{s_j}
\left(\ov{s}_i\ov{s}_j\otimes \ov{s}_i\ov{s}_j - \ov{s}_j\ov{s}_i\otimes \ov{s}_j\ov{s}_i\right) \in (R\#M)\otimes I + I \otimes (R\#M).$$
For $i<m-1$, we have ${\sigma_{s_i}}e_i^m = e_{i+1}^m$ and ${\sigma_{s_i}}e_{i+1}^m = e_{i}^m$, while ${\sigma_{s_i}}e_j^m = e_j^m$ if $j\not\in\{i, i+1\}$.
Hence 
\begin{equation}(\sigma_{s_i}\otimes \sigma_{s_i})J_{s_{i+1}}
= (e_{i}^m \otimes e_{i+2}^m)(J)
= (\sigma_{s_{i+1}}\otimes \sigma_{s_{i+1}})J_{s_{i}}
\end{equation}
In particular, 
$(\sigma_{s_i}\sigma_{s_{i+1}}\otimes \sigma_{s_i}\sigma_{s_{i+1}})(J_{s_{i}}) = J_{s_{i+1}}$ and $
(\sigma_{s_{i+1}}\sigma_{s_{i}}\otimes \sigma_{s_{i+1}}\sigma_{s_{i}})(J_{s_{i+1}})	= J_{s_i}$. 
Thus,
\begin{eqnarray}\label{eq:technical1}
	J_{s_i}(\sigma_{s_i}\otimes \sigma_{s_i})(J_{s_{i+1}})(\sigma_{s_i}\sigma_{s_{i+1}}\otimes \sigma_{s_i}\sigma_{s_{i+1}})(J_{s_{i}})
	&=&
	J_{s_i}(\sigma_{s_i}\otimes \sigma_{s_i})(J_{s_{i+1}}) J_{s_{i+1}}\\
\label{eq:technical2}
J_{s_{i+1}}
(\sigma_{s_{i+1}}\otimes \sigma_{s_{i+1}})(J_{s_{i}})
(\sigma_{s_{i+1}}\sigma_{s_{i}}\otimes \sigma_{s_{i+1}}\sigma_{s_{i}})(J_{s_{i+1}})	&=&	
	J_{s_{i+1}}(\sigma_{s_{i}}\otimes \sigma_{s_{i}})(J_{s_{i+1}}) J_{s_i}
\end{eqnarray}
Since $J$ is central,   the values in (\ref{eq:technical1}) and (\ref{eq:technical2}) are equal, call it $\gamma$. Then
$$
	\Delta(\ov{s}_{i}\ov{s}_{i+1}\ov{s}_{i})
	= J_{s_1}(\ov{s}_{i}\otimes \ov{s}_{i}) J_{s_{i+1}} (\ov{s}_{i+1}\otimes \ov{s}_{i+1})J_{s_1}(\ov{s}_{i}\otimes \ov{s}_{i})	
	= \gamma(\ov{s}_{i}\ov{s}_{i+1}\ov{s}_{i}
	\otimes \ov{s}_{i}\ov{s}_{i+1}\ov{s}_{i})
$$
and similarly $		
\Delta(\ov{s}_{i+1}\ov{s}_{i}\ov{s}_{i+1})
= \gamma(\ov{s}_{i+1}\ov{s}_{i}\ov{s}_{i+1}
\otimes \ov{s}_{i+1}\ov{s}_{i}\ov{s}_{i+1}).
$
Hence $\Delta(\ov{s}_{i}\ov{s}_{i+1}\ov{s}_{i}-\ov{s}_{i+1}\ov{s}_{i}\ov{s}_{i+1})$ is an element of $I\otimes (R\#M) + (R\#M) \otimes I$.
Defining $S(\ov{s}_i)=\ov{s}_i$, for all $i$, we obtain that $H=(R\# M)/I$ is a Hopf algebra. 
Note that $R\# M$ is graded by the length of words in $M$. Since elements of the form $\ov{s}_i\ov{s}_{j}-\ov{s}_j\ov{s}_{i}$ and  $\ov{s}_{i}\ov{s}_{i+1}\ov{s}_{i}-\ov{s}_{i+1}\ov{s}_{i}\ov{s}_{i+1}$ are homogeneous of degree $2$ and $3$ respectively, the only possible  $a\in R\cap I$ must be of the form  $a=\sum_{i=1}^{m-1} a_i(\ov{s}_i^2-t_i)$. Since words in $M$ form an $R$-basis of $R\# M$, we must have $a_i =0$, for all $i$, hence $R\cap I=0$ and $R$ embeds into $H=(R\# M)/I$ as Hopf subalgbra of $H$. Actually, $R$ is a normal Hopf subalgebra and  $HR^+=R^+H$ is a  Hopf ideal of $H$, where $R^+ = R\cap \Ker(\epsilon)$. 

(1) Since $s_1,\ldots, s_{n-1}$ generate $\Sigma_m$, we choose for each $w\in \Sigma_m$ a representation 
$w = s_{i_1}\cdots s_{i_{k}}$, with $k$ minimal. We will write $\ov{w}= \ov{s}_{i_1}\cdots \ov{s}_{i_{k}}$, for the corresponding element in $R\# M$. Since the $s_i$ satisfy $s_is_{i+1}s_i = s_{i+1}s_is_{i+1}$ and $s_is_j = s_js_i$,  for $|i-j|>1$, any word in $\ov{s}_1, \ldots, \ov{s}_{m-1}$ reduces to an element $\ov{w}$ modulo $I$. Thus   $\{ \ov{w}+ I : w\in \Sigma_m\}$  generates $H$ as left $R$-module. On the other hand an element of the form $\sum_{w\in \Sigma_m} a_w \ov{w} $, with $a_w\in R$ belongs to $I$ if and only if $a_w=0$, for any $w\in \Sigma_m$. Thus $H$ is free over $R$ with basis $\{ \ov{w} : w\in \Sigma_m\}$ and  $\dim(H)=\dim(R)^m m!$ follows. 

(2) Let $w,v \in \Sigma_m$ then there exists a unique element $\gamma(w,v)\in R$ such that $\ov{w}\:\ov{v} = \gamma(w,v)\ov{wv}$. More precisely, let $w= s_{i_1}\cdots s_{i_k}$ and $v= s_{j_1}\cdots s_{j_l}$  be the chosen minimal presentations of $w$ and $v$. If $wv$ has length $k+l$, then  set $\gamma(w,v)=1$. If the length of $wv$ is less than $k+l$, then some reduction occurred, which means that some relations of the type $s_i^2=\id$ must have reduced the length of  $wv$. In $H$, the reduction takes the form of replacing $\ov{s}_i^2$ by $t_{s_i}\in R^\times$. Depending on where in $wv$ the replacement takes place, some automorphism of the form $\sigma_u$ is applied to shift  $t_{s_i}$ to the ``left'' of $\ov{wv}$. Hence  $\ov{w}\:\ov{v} = \gamma(w,v) \ov{wv}$, where $\gamma(w,v)$ is a product of  elements of the form $\sigma_u(t_{s_i})$, for some $u\in \Sigma_m$ and some $i$.

(3) Since $J_{s_i}$ is central in $R\otimes R$, also $t_{s_i} = \mu_R(J_{s_i})$ is central in $R$. Furthermore, as $(\epsilon_B\otimes \id_B)(J)=1$, also $\epsilon(t_{s_i})=1$ and therefore $\epsilon(\sigma_u(t_{s_i}))=1$. Thus $\epsilon(\gamma(w,v))=1$, for any $w,v\in \Sigma_m$.

(4) If $B$ is semisimple, then $R$ is semisimple and the antipode of $R$ is an involution.  Since $S^2(\ov{w})=\ov{w}$, for $w\in \Sigma_m$, the antipode of $H$ is an involution and $H$ is semisimple.  
Let $\int_B$ be the integral of $B$, then  $\int = \int_B\otimes \cdots \otimes \int_B\in R=B^{\otimes m}$ is an integral for $R$. Note that $\int$ is invariant under the action of $\Sigma_m$. Hence, for $v\in \Sigma_m$, we have 
\begin{equation}
\ov{v}\int\sum_{w\in \Sigma_m}\ov{w} = \sigma_v(\int)\sum_{w\in \Sigma_m} \gamma(v,w)\ov{vw} =  \sum_{w\in \Sigma_m}\int\epsilon(\gamma(v,w)) \ov{vw} = \int \sum_{w\in \Sigma_m} \ov{w}.
\end{equation}
\end{proof}

\section{Generalized Kac-Paljutkin algebras}
The group algebra of a finite cyclic group provides a twist that satisfies all the requirements of the last Theorem.  Let  $n,m\geq 2$, $q$ a primitive $n$th root of unity and $\Z_n=\langle x : x^n=1\rangle$ the cyclic group of order $n$. For each $0\leq k < n$ set 
$$ e_k = \frac{1}{n} \sum_{i=0}^{n-1} q^{-ik} x^i.$$
Then $\{e_0, \ldots, e_{n-1}\}$ is a complete set of orthogonal idempotents of $B=\K \Z_n$.

\begin{lemma}\label{lem:cyclic}
Let $B=\K\Z_n$. Then 
$J=\sum_{k=0}^{n-1} e_k\otimes x^k \in B\otimes B$ is a twist for $B$ satisfying (\ref{eq:strong_twist}) and (\ref{eq:superstrong}), such that $(S\otimes S)(J)=J$ and $(\id\otimes S)(J)=J^{-1} = (S\otimes \id)(J)$ hold.
\end{lemma}
\begin{proof} It is well-known that $J$ is a twist for $B$.
We calculate $(S\otimes S)(J) = \frac{1}{n}\sum_{i,k=0}^{n-1} q^{-ik} x^{-i} \otimes x^{-k} = J$. 
Moreover,
\begin{equation}
(\id\otimes S)(J) J = \sum_{k,l=0}^{n-1} e_ke_l \otimes x^{l-k} = \sum_{k=0}^{n-1} e_k \otimes 1_B = 1_B\otimes 1_B,
\end{equation} shows $(\id \otimes S)(J)=J^{-1}$.  Since $S^2=\id$ and $(S\otimes S)(J^{-1})=J^{-1}$, also $(S\otimes \id)(J)=J^{-1}$ holds.  Next we  show that 
 $$J' =(e_1^2\otimes e_2^2)(J) = \sum_{k=0}^{n-1} e_k \otimes 1_B \otimes 1_B \otimes x^k =
 \sum_{k=0}^{n-1} x^k \otimes 1_B \otimes 1_B \otimes e_k$$
 is a twist for $B\otimes B$:
\begin{eqnarray*}
(\Delta_{B\otimes B} \otimes \id_{B\otimes B})(J')(J'\otimes 1_{B\otimes B})
&=& 
\sum_{k,l} (\Delta_{B\otimes B}(x^k\otimes 1)\otimes 1\otimes e_k) (e_l\otimes 1\otimes 1 \otimes x^l \otimes 1\otimes 1)\\
&=& \sum_{k,l} x^ke_l \otimes 1 \otimes x^k \otimes x^l \otimes 1 \otimes e_k\\
&=&
\sum_{k,l} e_l \otimes 1 \otimes x^k \otimes x^l \otimes 1 \otimes x^le_k\\
&=&
\sum_{k,l} \left(e_l \otimes 1 \otimes \Delta_{B\otimes B}(1\otimes x^l)\right)  (1\otimes  1 \otimes x^k \otimes 1 \otimes 1\otimes e_k)\\
&=& (\id_{B\otimes B} \otimes \Delta_{B\otimes B}(J')(1_{B\otimes B} \otimes J')
\end{eqnarray*}
where we use $x^ke_l =q^{kl}e_l$ and $q^{kl}e_k = x^le_k$.
Lastly  we will verify equation (\ref{eq:superstrong}). For this consider 
\begin{eqnarray*}
(e_1\otimes e_2)(J) (e_2\otimes e_1)(J) (J\otimes J) 
&=&
\sum_{k,l,s,t} (e_k\otimes x^l \otimes e_l \otimes x^k)(e_s\otimes x^s\otimes e_t \otimes x^t)\\
&=&\sum_{k,l} e_k\otimes x^{k+l} \otimes e_l \otimes x^{k+l}\\
&=& \sum_{k,j} e_k \otimes x^{j} \otimes e_{j-k} \otimes x^{j} 
= \Delta_{B\otimes B}(J),
\end{eqnarray*}
because
$ \Delta_B(e_j)=\sum_{k} e_k \otimes e_{j-k}$ holds.
\end{proof}

From  Theorem \ref{thm:HopfAlgebra} and Lemma \ref{lem:cyclic}, we deduce that there exists a non-trivial semisimple Hopf algebra 
$H_{n,m} = \left(\K\Z_n\right)^{\otimes m} \#_\gamma \Sigma_m$ of dimension $n^m m!$, for any integers $n,m\geq 2$ (provided $\K$ contains a primitive $n$th root of unity).

\begin{theorem}[Generalized Kac-Paljutkin algebras]
For any $n,m \geq 2$, $\K$ a field of characteristic zero with primitive $n$th root of unity $q$, there exists a semisimple Hopf algebra $H_{n,m} = \K\Z_n^{\otimes m} \#_\gamma \Sigma_m$ of dimension $n^m m!$ that is an extension of   $\K \Sigma_m$, by $\K\Z_n^{\otimes m}$, which  is generated by elements $x_1,\ldots, x_m, z_1,\ldots, z_{m-1}$ such that 
$$x_i^n=1, \qquad x_ix_j=x_ix_j, \qquad z_k x_i = x_{s_k(i)} z_k, \qquad \forall \:i,j,k,$$ where $s_k$ is the transposition $s_k=(k, k+1)$. Furthermore, for all $k,l$ with $|k-l|>1$:
\begin{equation}
z_kz_l=z_lz_k, \qquad z_kz_{k+1}z_k = z_{k+1}z_kz_{k+1},  \qquad   z_k^2=\frac{1}{n} \sum_{i,j=0}^{n-1} q^{-ij} x_{k}^i x_{k+1}^j.
\end{equation}
The elements $x_i$ are group-like, while 
\begin{equation}
\Delta(z_k) =  \left(\frac{1}{n} \sum_{i,j=0}^{n-1} q^{-ij} x_{k}^i \otimes  x_{k+1}^j \right) (z_k\otimes z_k), \quad \epsilon(z_k)=1, \quad S(z_k)=z_k.
\end{equation}
\end{theorem}

As before, we  write $t_k:=\frac{1}{n} \sum_{i,j=0}^{n-1} q^{-ij} x_{k}^i x_{k+1}^j$ and $J_k := \frac{1}{n} \sum_{i,j=0}^{n-1} q^{-ij} x_{k}^i \otimes  x_{k+1}^j $. Since $(\id\otimes S)(J_k) = J_k^{-1}$, we also have 
$$t_k^{-1} = \frac{1}{n} \sum_{i,j=0}^{n-1} q^{-ij} x_{k}^i x_{k+1}^{-j}.$$


\begin{remark}  Let $G:=\Z_n^m$. Then $\K G \simeq \K\Z_n^{\otimes m}$. By construction,  $\K G$ is a Hopf subalgebra of $H:=H_{n,m}$ and $\K \Sigma_m$ is a Hopf quotient algebra of $H$.
Since the ring of $\K \Sigma_m$-coninvariants of $H$ is $\K G$ and since $\K G \simeq \K^G$ as Hopf algebras\footnote{More precisely, for any $\alpha \in \{0,\ldots, n-1\}^m$ set $e_{\alpha} :=   \frac{1}{n^m} \prod_{k=1}^{m} \sum_{i=0}^{n-1} q^{-i\alpha_k} x_k^i \in \K[G].$ Then $\{ e_{\alpha} : \alpha \in \{0,\ldots, n-1\}^m\}$ is a complete set of orthogonal idempotents of $\K[G]$. Equally, the projections $p_\alpha \in \K^G$ defined by $p_\alpha (x^\beta) = 1$ if and only if $\alpha=\beta$ and $0$ otherwise, for any $\beta\in \{0,\ldots, n-1\}^m$, where $x^\beta := x_1^{\beta_1}\cdots x_m^{\beta_m}$, form a complete set of orthogonal idempotents of $\K^G$. Defining $\psi: \K[G]\to \K^G$ with $\psi(e_\alpha)=p_\alpha$ yields an algebra homomorphism, which is actually a Hopf algebra isomorphism.}, as  $G$ is finite Abelian, the Hopf algebra $H$ is an extension of $\K \Sigma_m$ by $\K^G$ and hence equivalent to to an element of  $\operatorname{Opext}(\K \Sigma_m, \K^G)$, the set of equivalence classes of all extensions of $\K\Sigma_m$ by $\K^G$, associated with a matched pair $(\Sigma_m, G)$. Here the matched pair is given by $\Sigma_m$ acting on $G$ by permutations of its components and $G$ acting trivially on $\Sigma$.  
 The set  $\operatorname{Opext}(\K \Sigma_m, \K^G)$ forms a group with identity element being the group ring $\K[G\rtimes \Sigma_m]$ of the semidirect product of $G$ and $\Sigma_m$. Since  $H_{n,m}$ is not cocommutative, $H_{n,m}$ is equivalent to a non-trivial element of $\operatorname{Opext}(\K\Sigma_m, \K^G)$.
By \cite[Proposition 1.5]{Masuoka2002}, elements of $\operatorname{Opext}(\K\Sigma_m, \K^G)$ are equivalent to a bicrossed products $\K^G {\#^\tau_\gamma}\: \K\Sigma_m$, for cocycles $\gamma: \Sigma_m \times \Sigma_m \to \K^G$ and $\tau: G\times G\to \K^{\Sigma_m}$. From the multiplication of $H_{n,m}$, we see that $\tau$ is constant $1_{K^{\Sigma_m}}$, while $\gamma$ depends on the chosen twist $J$.

For $m=2$, the elements of $\operatorname{Opext}(\K\Sigma_2, \K^{\Z_n\times \Z_n })$ have been classified in \cite[Theorem 2.1]{Masuoka1997} and are in bijective correspondence with all $n$th roots of unity in $\K$.  Thus, the Hopf algebras $H_{n,2}$ had  already been discovered in \cite{Masuoka1997}, prior to \cite{PanseraPhD}. In particular, for each root of unity $q$, there exists a twist $J$ (as defined above) that determines the comultiplication of such an extension via $\Delta(z)=J(z\otimes z)$.  The cocycle $\gamma $ is determined by the element $t=\mu(J)$, where $\gamma(s,s)=t$ for the generator $s$ of $\Sigma_2=\Z_2$. 

For $m\geq 3$, it would be interesting to classify $\operatorname{Opext}(\K\Sigma_m, \K^G)$ in general, whereas \(H_{n,m}\) provides a non-trivial example of it. Again, for each $n$-th root of unity $q$, we use the same twist $J$ of $B=\K\Z_n$ as above. For each Coxeter generator $s_{i}$ of $\Sigma _{m}$, we deform the comultiplication of $z_i = \overline{s_i}$ by defining $\Delta(z_i)=J_i(z_i\otimes z_i)$, where $J_i = (e_i^m\otimes e_{i+1}^m)(J) \in R\otimes R$ for $R=B^{\otimes m}=\K G$ means that $J$ is placed into the $i$-th and $(i+1)$-st components, respectively. As Example \ref{exa:m3} illustrates, the cocycle $\gamma $ is determined by $t=\mu(J)$, which is embedded suitably into $R$ such that $\gamma(s_i,s_i)=t_i:=\mu(J_i)$ and $\gamma(s_i,s_j)=1$ for $i\neq j$. Hence, the non-triviality of $H_{n,m}$ as an element of $\operatorname{Opext}(\K\Sigma_m, \K^G)$ stems essentially from the non-triviality of the Kac-Paljutkin algebra $H_{n,2}$ and has just been placed into an $m$-fold tensor product of $\K\Z_n$.

Note that it is shown in \cite{Sebastian} that $H_{n,m}$ is isomorphic, as an algebra, to the group ring $\K[\Z_n \wr \Sigma_m]$ of the wreath product of $\Z_n$ and $\Sigma_m$. Since the wreath product $\Z_n\wr \Sigma_m$ is the semidirect $\Z_n^m \rtimes \Sigma_m$, the Hopf algebra $H_{n,m}$ is isomorphic as an algebra to the  smash product $\K^G \#\: \K\Sigma_m$ (without cocycle).  However, due to the fact that $H_{n,m}$ is not cocommutative, this isomorphism is not an isomorphism of Hopf algebras.

\end{remark}

\begin{exa}[Case $m=2$]
 Write $x$ for the generator of $B=\K\Z_n$ and write the elements of $R=B^{\otimes 2} = \K \Z_n \otimes \K\Z_n$ as linear combinations of elements $x^iy^j$, where $x$ stands for $x\otimes 1$ and $y$ for $1\otimes x$.  Since $m=2$, we have only one transposition, which is $s=(12)$. Hence $\sigma_s(x)=y$, $\sigma_s(y)=x$ and the twist $J=\sum_{j=0}^{n-1} e_j\otimes x^j$ for $B$ becomes the twist 
$$J_s = \sum_{j=0}^{n-1} e_j \otimes y^j = \frac{1}{n} \sum_{i,j=0}^{n-1} q^{-ij} x^i \otimes y^j $$
for $R$. 
Instead of $\ov{s}$ write $z$ for the generator of the monoid  $M=\langle \ov{s}\rangle$. 
Set $t:= \mu(J_s) = \frac{1}{n}\sum_{i,j=0}^{n-1} q^{-ij} x^iy^j$. The Hopf algebra $H_{n,2}=(R\#M)/\langle z^2-t\rangle$ has therefore generators $x,y,z$ subject to 
$$ xy=yx, \quad x^n=1=y^n, \quad zx=yz, \quad zy=xz, \quad z^2= t.$$
While $x$ and $y$ are group-like, $$\Delta(z) = \left( \frac{1}{n} \sum_{i,j=0}^{n-1} q^{-ij} x^i \otimes y^j \right)(z\otimes z).$$
A basis for $H_{n,2}$ are the monomials  $\{x^iy^jz^k : 0\leq i,j<n, k\in\{0,1\}\}$.
Hence $H_{n,2}$ is precisely the generalized Kac-Paljutkin algebra $H_{2n^2}$ considered by Pansera.
In particular, $H_{2,2}$ is the Kac-Paljutkin Hopf algebra itself.
Note that  $H_{n,2} = \K[\Z_n\times \Z_n] \#_\gamma \Sigma_2$, for $\gamma:\Sigma_2 \times \Sigma_2 \to R^\times$ given by $$\gamma(\id,\id)=\gamma(\id, (12)) = \gamma((12),\id) = 1, \qquad \gamma((12),(12)) = t.$$

\end{exa}

\begin{exa}[Case $m=3$]\label{exa:m3} Let $q$ be a primitive third root of unity.
Write $x$ for the generator of $B=\K\Z_n$ and $R=B^{\otimes 3} = \K \Z_n \otimes \K\Z_n\otimes \K\Z_n$. Elements of $R$ are  linear combinations of  monomials $x_1^ix_ 2^jx_3^k$, where $x_1$ stands for $x\otimes 1\otimes 1$, $x_2$ stands for  $1\otimes x\otimes 1$ and  $x_3$ stands for  $1\otimes 1\otimes x$.  The transpositions $s_1=(12)$ and $s_2=(23)$ are generators of $\Sigma_3$.  Write $\sigma_i = \sigma_{s_i}$ for the corresponding automorphism of $R$. The twist $J=\sum_{j=0}^{n-1} e_j\otimes x^j$ for $\K \Z_n$ yields two twists
$$J_{s_1} = \frac{1}{n} \sum_{i,j=0}^{n-1} q^{-ij} x_1^i \otimes x_2^j \qquad \mbox{ and } \qquad 
J_{s_2} = \frac{1}{n} \sum_{i,j=0}^{n-1} q^{-ij} x_2^i \otimes x_3^j$$
for $R$. 
Set $z_1:= \ov{s_1}$ and  $z_2:= \ov{s_2}$ for the generator of the free monoid  $M=\langle \ov{s_1}, \ov{s_2}\rangle$. 
The Hopf algebra 
\begin{equation} 
H_{n,3}=(R\#M)/\langle z_1^2-t_{1}, z_2^2-t_{2}, z_1z_2z_1-z_2z_1z_2\rangle = \left(\K \Z_n\right)^{\otimes 3} \#_\gamma \Sigma_3\end{equation}
has therefore generators $x_1,x_2,x_3,z_1,z_2$ subject to $x_ix_j=x_jx_i$ and $x_i^n=1$, for all $i,j$ and
$$ z_1x_1=x_2z_1, \quad z_1x_2=x_1z_1, \quad z_1x_3=x_3z_1$$
$$ z_2x_1=x_1z_2, \quad z_2x_2=x_3z_2, \quad z_2x_3=x_2z_2$$
$$z_1z_2z_1 = z_2z_1z_2, \qquad z_1^2=t_1:=\frac{1}{n} \sum_{i,j=0}^{n-1} q^{-ij} x_1^ix_2^j, \qquad z_2^2=t_2:=\frac{1}{n}\sum_{i,j=0}^{n-1} q^{-ij} x_2^ix_3^j.$$
While the $x_i$ are group-like, we have 
$$\Delta(z_1) = \left( \frac{1}{n} \sum_{i,j=0}^{n-1} q^{-ij} x_1^i \otimes x_2^j \right) (z_1\otimes z_1) \qquad
\Delta(z_2) = \left(\frac{1}{n} \sum_{i,j=0}^{n-1} q^{-ij} x_2^i \otimes x_3^j \right) (z_2\otimes z_2).$$
 $H_{n,3}$ has dimension $6n^3$ and basis $ \left\{x_1^ix_2^jx_3^k \ov{w} : 0\leq i,j,l<n, \:\ov{w} \in \{1, z_1, z_2, z_1z_2, z_2z_1, z_1z_2z_1\} \right\}$.
The $2$-cocycle $\gamma$ is determined by the multiplication in $H_{n,3}$ and given by $\gamma(\id, w)=\gamma(w,\id)=1$, for $w\in \Sigma_3$, and the following table. Note that  $t_2=\sigma_1\sigma_2(t_1)=t^{\sigma_2\sigma_1}$ and 
$\sigma_1(t_2)=\sigma_2(t_1) =  \frac{1}{n}\sum_{i,j=0}^{n-1} q^{-ij} x_1^ix_3^j =t^{\sigma_2}$. The values of the $2$-cocycle $\gamma$ can be calculated and are as follows:
\[\begin{array}{r|ccccc}
	\gamma	 & s_1 					& s_2	& s_1s_2 	& s_2s_1  & s_1s_2s_1\\\hline
s_1		 & t_1					&  	1	& 	t_1	& 	1	&  	t_1		\\
s_2		 & 	1					&  	t_2	& 	1	& 	t_2	&  	t_2	\\
s_1s_2	 & 	1					&  	t_3	& 	t_1	& 	t_1t_3	& t_1t_3			\\
s_2s_1	 & 	t_3			&  	1	& 	t_2t_3	& t_2	& t_2t_3 			\\
s_1s_2s_1 & 	t_2	&  	t_1	& 	t_2t_3	& t_1t_3 		&  t_1t_2t_3
\end{array}\]
\end{exa}
 \begin{remark}
 We see, that $H_{n,2}$ embeds into $H_{n,3}$ by sending $x,y, z$ of $H_{n,2}$ to the corresponding elements $x_1,x_2,z_1$ in $H_{n,3}$. More generally,  for $n,m\geq 2$, the algebra $R=\K\Z_n^{\otimes m}$ embeds into $R' = \K\Z_n^{\otimes m+1}$ by sending   $a_1\otimes \cdots \otimes a_m$ into $a_1\otimes \cdots \otimes a_m \otimes 1_B$. Similarly,  $\Sigma_m$ can be considered a subgroup of $\Sigma_{m+1}$, where we identify the generators $s_1,\ldots, s_{m-1}$ with the first $m-1$ generators of $\Sigma_{m+1}$.
 Hence, if $x_1,\ldots, x_m$ denotes the basis of $R=\K\Z_n^{\otimes m+1}$ and $z_1,\ldots, z_{m-1}$ the remaining algebra generators of $H_{n,m}$ and if $x_1', \ldots, x_{m+1}'$ denote the basis of $R'$ and $z_1',\ldots , z_{m+1}'$ the remaining generators of $H_{n,m+1}$, then mapping $x_i$ to $x_i'$ and $z_i$ to $z_i'$ yields an algebra embedding of $H_{n,m}$ into $H_{n,m+1}$, which is also an embedding of Hopf algebras. 
 \end{remark}

 \begin{remark}\label{rem:comult}
 Let $H=H_{n,m}$ and $w\in \Sigma_m$. We claim that the comultiplication of $\ov{w}$ in $H$ looks like $\Delta(\ov{w}) = J(w) (\ov{w}\otimes \ov{w})$ for an invertible element $J(w)\in R\otimes R$ with $\epsilon(\mu_R(J(w)))=1$. We will prove this by induction on the length of the chosen representation $w=s_{i_1}\ldots s_{i_k}$ in the generators $s_1,\ldots s_{m-1}$. For $k=1$, we have $w=s_i$, for some $i$ and by definition $\Delta(z_i)=J_i(z_i\otimes z_i)$ for $z_i=\ov{s_i}$ and $J_i = J_{s_i} = (e_i^m\otimes e_{i+1}^m)(J)$, which is invertible, since $J$ is invertible. Moreover, $\mu_R(J_i)=e_{i, i+1}^m(J) = t_i$ and $\epsilon(t_i)=1$. Now suppose that $w=s_iv$ has length greater than $1$, where $v\in \Sigma_m$ and there exists $J(v)\in R\otimes R$ such that $\Delta(\ov{v})=J(v)(\ov{v}\otimes \ov{v})$. Then 
\begin{equation}
\Delta(\ov{w}) = J_i (z_i\otimes z_i) J(v) (\ov{v}\otimes \ov{v}) = J_i J(v)^{\sigma_i} (\ov{w} \otimes \ov{w}),
\end{equation}
where we denote by $J(w)^{\sigma_i} = (\sigma_i\otimes \sigma_i)(J(w))$. Setting $J(w)=J(s_iv)=J_iJ(v)^{\sigma_i} $ proves our claim, since also $\mu_R(J(w)) = t_i \sigma_i( \mu_R(J(v)))$ has counit $1$.

Thus, given a subgroup $N$ of $\Sigma_m$, we can consider the Hopf subalgebra of $H$ generated by $R$ and $\{\overline{w} : w\in N\}$,   which we shall denote by $R\#_\gamma N$. Since for any $w\in N$, $\Delta(\overline{w})=J(w)(\overline{w}\otimes \overline{w}) \in (R\#_\gamma N) \otimes (R\#_\gamma N)$ we conclude that the subalgebra $R\#_\gamma N$ is a semisimple Hopf subalgebra of $H$.

 \end{remark}
 
 \begin{exa}\label{exa:cyclic}
 Consider $H_{n,m}=R\#_\gamma \Sigma_m$ with generators  as above and consider $\theta$, the product of all $z_i's$, i.e. $\theta =z_1 \cdots z_{m-1}$. Let $H$ be the subalgebra of $H_{n,m}$  generated by $R$ and $\theta$. We claim that $H$ is equal to $R\#_{\gamma} \langle s\rangle$, where $s=(12\cdots m)$ is the cycle of length $m$.  Clearly, $\theta = \ov{s}$ and we claim that $\theta^k = \left(\prod_{i=1}^{k-1} \gamma(s^i,s)\right) \ov{s^k}$, for $2\leq k\leq m$, since inductively, if $\theta^k = c_k \ov{s^k}$, for some $c_k\in R^\times$, then 
 $\theta^{k+1} = c_k \ov{s^k} \: \ov{s} = c_k\gamma(s^k,s) \ov{s^{k+1}}$.
Thus $H=R\#_{\gamma} \langle s\rangle$,  as $\left(\prod_{i=1}^{k-1} \gamma(s^i,s)\right) \in R^\times$.
 Note that $\theta^m = \left(\prod_{i=1}^{m-1} \gamma(s^i,s)\right) =:t \in R^\times$. Hence $H\simeq R[\theta;\sigma]/\langle \theta^m-t\rangle$, where $\sigma = \sigma_s$.
In particular, $H$ is generated by $x_1,\ldots, x_m$ and $\theta$ subject to  
 $\theta x_i = x_{s(i)} \theta$ and $\theta^m=t$. While $\Delta(\theta) = J(s)(\theta\otimes \theta)$, with $J(s)$ as in Remark \ref{rem:comult}. 
Note that  $ \theta^{m-1}  = \tilde{t}\: \ov{s}^{m-1}$, for $\tilde{t} = t \gamma(s^{m-1},s)^{-1}$.  Since $s^{-1} = s^{m-1}$ we obtain
 \begin{equation}
 \theta^{m-1}  = \tilde{t}  \:\ov{s} ^{ -1} =\tilde{t} \:z_{m-1} \cdots z_1 = \tilde{t} S(\theta).
 \end{equation}
 Hence $S(\theta) = \tilde{t}^{-1} \theta^{m-1} = t^{-1} \gamma(s^{m-1},s) \theta^{m-1}$.
 The dimension of $H$ is $mn^m$. For $m=2$, we obtain $H'=H_{n,2}$.
 \end{exa}

\section{Actions of the generalized Kac-Paljutkin algebra}
As before, let $H=H_{n,m} = R\#_\gamma \Sigma_m$, with $R=B^{\otimes m}$ and $B=\K\Z_n$, for some $n,m\geq 2$.

The aim of this section is to define an action of $H$ on a non-commutative algebra $A$ such that the action does not factor through a group algebra - in other words, the action is inner-faithful. To achieve this, we first examine $m$-dimensional simple $H$-modules $V$ (Lemma \ref{lem:simple}) in order to define a quantum polynomial algebra $A$ generated by the basis $u_1,\ldots, u_m$ of $V$, such that $V$ is an $H$-submodule of $A$. After introducing commutation relations between the generators $u_i$, we obtain a Noetherian domain with an inner-faithful $H$-action (Theorem \ref{thm:action}).

 Let $q$ be the primitive $n$th root of unity used in the definition of the twist $J$ of $B$. Since $\K \Sigma_m$ is a factor ring of $H$, any irreducible representation over $\Sigma_m$ is a simple module over $H$.
As above, let  $x_1, \ldots, x_m$ and $z_1,\ldots, z_{m-1}$ be the generators of $H$. For each transposition $s_i = (i,  i+1)$ we have $z_i=\ov{s}_i$ and set $\sigma_i := \sigma_{s_i}$. 

Assume that there exists a square root $p$ of $q$.
Given $0\leq a,b <n$ we define a left $H$-module structure on an $m$-dimensional vector space $V_{a,b}$ with basis $u_1,\ldots, u_m$ by letting $x_i$ and $z_k$ respectively act as $m\times m$-matrices:
\begin{equation}
X_i = \left[\begin{array}{cccccc}
q^b &   	   		&		&		&		& \\
	    &		    \ddots & 		&		&		&\\	
	    &		   		& q^a	 	&  		&		&	 \\
	   &   	         		&		&		&\ddots	& \\
	    &		    		 & 		&		&		& q^b
\end{array}\right]
\qquad 
Z_k = \left[\begin{array}{cccccc}
p^{b^2} &   	   		&		&		&		& \\
	    &		    \ddots & 		&		&		&\\	
	    &		   		& 0	 	& 1 		&		&	 \\
	    &		  		& q^{ab}  &  0  		&		&\\	
	   &   	         		&		&		&\ddots	& \\
	    &		    		 & 		&		&		& p^{b^2}
\end{array}\right]
\end{equation}
where $X_i$ has on its diagonal the value $q^b$, except on the $i$th row, where it has $q^a$. The matrix $Z_k$
 has on its diagonal the value $p^{b^2}$ except for the $k$th and $k+1$st rows (resp. colums), where one has a $2\times 2$ block of the form $\left[\begin{array}{cc} 0 & 1 \\ q^{ab} & 0\end{array}\right]$. It is clear, that if $1\leq k<l<m$ such that $|k-l|>1$, then $Z_k$ and $Z_l$ commute, since their $2\times 2$-block is in different rows and columns. Hence $z_kz_l\cdot u_j=z_lz_k\cdot u_j$ holds, for $|k-l|>1$.

For the braid relation $z_kz_{k+1}z_k = z_{k+1}z_kz_{k+1}$ it is enough to look at the $3\times 3$-submatrix of the rows and columns of index $k,k+1$ and $k+2$ of the product $Z_kZ_{k+1}Z_k$, i.e. 
\begin{equation}
\left[\begin{array}{ccc}
0 	& 1 		&	0\\
q^{ab} & 0  	&	0\\	
0 & 0  	&	p^{b^2}
\end{array}\right]
\left[\begin{array}{ccc}
p^{b^2} 	& 0 		&	0\\
0 & 0  	&	1\\	
0 & q^{ab}  	&	0
\end{array}\right]
\left[\begin{array}{ccc}
0 	& 1 		&	0\\
q^{ab} & 0  	&	0\\	
0 & 0  	&	p^{b^2}
\end{array}\right]
= p^{b^2}
\left[\begin{array}{ccc}
0 	& 0 		&	1\\
0& q^{ab}   	&	0\\	
q^{2ab}&0  	&	0
\end{array}\right]
\end{equation}
Similarly, for $Z_{k+1}Z_kZ_{k+1}$ one obtains:
\begin{equation}
\left[\begin{array}{ccc}
p^{b^2} 	& 0 		&	0\\
0 & 0  	&	1\\	
0 & q^{ab}  	&	0
\end{array}\right]
\left[\begin{array}{ccc}
0 	& 1 		&	0\\
q^{ab} & 0  	&	0\\	
0 & 0  	&	p^{b^2}
\end{array}\right]
\left[\begin{array}{ccc}
p^{b^2} 	& 0 		&	0\\
0 & 0  	&	1\\	
0 & q^{ab}  	&	0
\end{array}\right]
= p^{b^2}
\left[\begin{array}{ccc}
0 	& 0 		&	1\\
0& q^{ab}   	&	0\\	
q^{2ab}&0  	&	0
\end{array}\right]
\end{equation}
It is not difficult to check that $Z_kX_i = X_{s_k(i)}Z_k$ holds for all $i$ and $k$. Moreover, using $p^2=q$,
\begin{equation}
Z_k^2 = \left[\begin{array}{llllll}
q^{b^2} &   	   		&		&		&		& \\
	    &		    \ddots & 		&		&		&\\	
	    &		   		& q^{ab} 	&  		&		&	 \\
	    &		  		&    &  q^{ab}  		&		&\\	
	   &   	         		&		&		&\ddots	& \\
	    &		    		 & 		&		&		& q^{b^2}
\end{array}\right] = \frac{1}{n}\sum q^{-st} X_k^s X_{k+1}^t,
\end{equation}
since the latter matrix is a diagonal matrix with 
$\frac{1}{n}\sum_{s,t} q^{-st+bs+bt} = q^{b^2}$ on its $i$th row, for $i\notin\{k,k+1\}$, 
and $\frac{1}{n}\sum_{s,t} q^{-st+as+bt} = q^{ab}$ on its $k$th and $k+1$st row. Thus, $V_{a,b}$ is a left $H$-module.

\begin{lemma}\label{lem:simple}
 $V_{a,b}$ is a simple left $H_{n,m}$-module, if $a\neq b \:(\mathrm{mod}\: n)$. 
\end{lemma}

\begin{proof}
Let $H=H_{n,m}$ and assume  $a\neq b\:(\mathrm{mod}\: n)$. 
Let $v=\sum_{i=1}^m \lambda_i u_i$ be a non-zero element of $v$ with $\lambda_k \neq 0$. Then $\lambda_k^{-1}(x_k-q^b)\cdot v = (q^a-q^b)u_k \neq 0$. Thus,   $u_k\in H\cdot v$. However, since $\Sigma_m$ acts transitively on $\{1,\ldots, m\}$, any basis element $u_j$ of $V_{a,b}$ can be obtained by multiplying $u_k$ by a suitable product of $Z_i's$. Hence $V_{a,b}$ is simple.
\end{proof}

\begin{lemma}
Let $m\geq 3$ and  $a\neq b \:(\mathrm{mod}\: n)$ and $a'\neq b' \:(\mathrm{mod}\: n)$. Then
 $V_{a,b} \simeq V_{a',b'}$ as left $H_{n,m}$-module if and only if  $(a,b) = (a',b')$ in $\Z_n^2$.
\end{lemma}
\begin{proof}
Suppose $f: V_{a,b} \to V_{a',b'}$ is an isomorphism of $H_{n,m}$-modules and $(\lambda_{ij})_{1\leq i,j\leq m}$ is a scalar matrix such that $f(u_i)=\sum_{j=1}^m \lambda_{ij} u_j'$, where $\{u_1,\ldots, u_m\}$ respectively $\{u_1', \ldots, u_m'\}$ is a basis for $V_{a,b}$ respectively $V_{a',b'}$. We will show that $(\lambda_{ij})$ is a diagonal matrix. Suppose there exists $i\neq j$ with $\lambda_{ij}\neq 0$. Let $l\neq i,j$ (using $m\geq 3$) and consider the following equations:
\begin{eqnarray}\label{eq:action}
\sum_{r=1}^m \lambda_{ir}q^b u_r' = q^b f(u_i) &=& x_j\cdot  f(u_i) = \lambda_{ij} q^{a'} u_j' +  \sum_{r\neq j}^m \lambda_{ir}q^{b'} u_r' \\
\sum_{r=1}^m \lambda_{ir}q^b u_r' = q^b f(u_i) &=& x_l\cdot f(u_i) = \lambda_{il} q^{a'} u_l' +  \sum_{r\neq l}^m \lambda_{ir}q^{b'} u_r'.
\end{eqnarray}
Comparing the coefficients of $u_j'$ on  both sides of each equation, we conclude 
 $b=a' \:(\mathrm{mod}\: n)$ and   $b=b' \:(\mathrm{mod}\: n)$, which contradicts the assumption $a'\neq b' \:(\mathrm{mod}\: n)$.  Hence $f(u_i)=\lambda_{ii} u_i'$ and $\lambda_{ii}\neq 0$, for all $i$. Thus 
 $\lambda_{ii}q^au_i' = q^af(u_i)=x_i \cdot f(u_i)=\lambda_{ii}q^{a'}u_i'$ and 
$ \lambda_{ii}q^bu_i' = q^bf(u_i)=x_j \cdot f(u_i)=\lambda_{ii}q^{b'}u_i'$, for any $j\neq i$ shows $a=a'$ and $b=b'$ modulo $n$.
The converse statement holds, since  the map $f:V_{a,b}\to V_{a',b'}$ with $f(u_i)=u_i'$ is an isomorphism of $H_{n,m}$-modules, in case $(a,b)=(a',b')$ in $\Z_n^2$.
 \end{proof}

Denote by $M_{m,a,b}\in M_{m}(\Z)$ the $m\times m$-matrix, such that the diagonal entries are all equal to $a$ and all other entries are equal to $b$. For example for $m=2$ and $m=3$ we have the matrices
\begin{equation}
M_{2,a,b} = \left[\begin{array}{cc} a & b \\ b & a \end{array}\right], \quad 
M_{3,a,b} = \left[\begin{array}{ccc} a & b& b \\ b & a &b \\ b & b& a \end{array}\right], \quad
\end{equation}
The determinant of such matrices are  $\det(M_{2,a,b}) = a^2-b^2$ and $\det(M_{3,a,b}) = a^3+2b^3-3ab^2$.
More generally, if $e_i$ denotes the canonical basis of $\K^m$, then $v_1=\sum_{i=1}^m e_i$ is an eigenvector of $M_{m,a,b}$ with eigenvalue $a+(m-1)b$ and the vectors $v_i := e_{i-1}-e_{i}$, for $2\leq i \leq n$ are eigenvectors with eigenvalue $a-b$. Hence,
\begin{equation}
\det(M_{m,a,b})  = (a+(m-1)b)(a-b)^{m-1}.
\end{equation}

 An $H$-module $V$ is inner-faithful if the only Hopf ideal that annihilates $V$ is the zero ideal. 

\begin{lemma}\label{lem:inner-faithful}  $V_{a,b}$ is an inner-faithful $R$-module if $\gcd((a+(m-1)b)(a-b)^{m-1}, n)=1$.
\end{lemma}

\begin{proof}
Let $I$ be a Hopf ideal of $R$. Since $R$ is a group algebra, there exists a subgroup $N\subseteq G=\Z_n^m$ such that $I = \K G (\K N)^+$ (see \cite[Lemma 4]{Farnsteiner} or \cite[p.329]{Passman}). Suppose $I\cdot V_{a,b}=0$. Then
$\omega= 1-x_1^{\alpha_1} \cdots x_m^{\alpha_m}\ \in I$, for some  $x_1^{\alpha_1} \cdots x_m^{\alpha_m}\in N$.  For any $1\leq j \leq m$ we obtain an equation
$\omega\cdot u_j = 0$, i.e. $1 = q^{\alpha_j a + \sum_{l\neq j} \alpha_l b}$. Thus,  $\alpha_j a + \sum_{l\neq j} \alpha_l b = 0 \: (\mathrm{mod}\: n),$ or in matrix notation $\alpha M_{m,a,b}  = 0 \: (\mathrm{mod}\: n)$. If $\gcd(\det(M_{m,a,b}), n)=1$, then $M_{m,a,b}$ is invertible modulo $n$. Thus $\alpha_i=0$, for all $i$, which means that $N$ is trivial and so is $I$.
\end{proof}

Next we  show that $H_{n,m}$ acts on a quantum polynomial algebra $A$  in $m$ variables, such that $V_{a,b}$ is its degree one space.

\begin{theorem}\label{thm:qpoly}
Let $a\neq b \:(\mathrm{mod}\: n)$,  $q$ a primitive $n$th root of unity in $\K$ and a square root $p$ of $q$. Set $\lambda = q^{b^2-ab}$ and $\mu = p^{b^2-a^2}$ and let $(r_{ij})$ be the $m\times m$-matrix such that $r_{ii}=1$, $ r_{ij} = \lambda^{j-i-1} \mu$ for $i<j$ and $r_{ij}=r_{ji}^{-1}$ for $j<i$.  Consider the quantum polynomial algebra in $m$ variables $u_1,\ldots, u_m$  subject to
$u_i u_j = r_{ij} u_ju_i$, for all $i,j$.
\begin{equation}
	A_{a,b} = \frac{\K\langle u_1,\ldots, u_m\rangle}{\langle u_iu_j - r_{ij}u_ju_i :i,j\rangle}.
\end{equation} 
Then $A_{a,b}$ is a left $H$-module algebra for the semisimple Hopf algebra $H=\K\Z_n^{\otimes m}\#_\gamma \Sigma_m$, where $V_{a,b}$ determines the $H$-action on the generators $u_i$, i.e.
$$x_i\cdot u_i = q^a u_i, \qquad x_j\cdot u_i = q^b u_i, \quad \mbox{ for } i\neq j$$
$$z_i\cdot u_i = q^{ab} u_{i+1}, \qquad z_i\cdot u_{i+1} = u_i, \qquad z_i \cdot u_j = p^{b^2} u_j, \qquad \mbox{ for } j \neq \{i,i+1\}.$$
\end{theorem} 

\begin{proof}
We claim that $A$ is a left $H$-module algebra. As before, $R=\K\Z_n^{\otimes m}$.  Since the group-likes $x_i$ act as scalars on $u_j$, we conclude that $A$ is an $R$-module algebra. Recall that 
\begin{equation}
	\Delta(z_i) = J_i (z_i \otimes z_i) = \frac{1}{n} \sum_{s,t} q^{-st} x_i^sz_i \otimes x_{i+1}^tz_i.
\end{equation}
Hence for $1\leq i< m$:
\begin{eqnarray*}
	z_i\cdot (u_iu_{i+1}) 
	&=&
	\frac{1}{n} \sum_{s,t} q^{-st} \left(x_i^sz_i\cdot u_i \right)\left( x_{i+1}^tz_i\cdot u_{i+1}\right)
	= q^{ab} q^{b^2} u_{i+1}u_i \\
	z_i\cdot (u_{i+1}u_i) &=&
	\frac{1}{n} \sum_{s,t} q^{-st} \left(x_i^sz_i\cdot u_{i+1} \right)\left( x_{i+1}^tz_i\cdot u_{i}\right)
	= q^{ab} q^{a^2} u_{i}u_{i+1} = q^{ab} q^{a^2} \mu u_{i+1}u_{i},
\end{eqnarray*}
where   $r_{i,i+1} = \mu=p^{b^2-a^2}$. Thus 
$z_i\cdot(u_iu_{i+1}) = q^{ab} q^{b^2}u_{i+1}u_i = q^{ab}\mu^2 q^{a^2} u_{i+1}u_i = z_i\cdot ( r_{i,i+1} u_{i+1}u_i).$

A similar calculation for $j>i+1$ shows 
\begin{eqnarray*}
z_i\cdot (u_iu_j) &=& p^{b^2} q^{ab+b^2} u_{i+1}u_j\\
z_i\cdot (u_ju_i) &=& p^{b^2} q^{2ab} u_{j}u_{i+1} = p^{b^2}q^{2ab} r_{j,i+1} u_{i+1}u_j
\end{eqnarray*}
Since $r_{ij}r_{j,i+1} = r_{ij}r_{i+1,j}^{-1} = \lambda = q^{b^2-ab}$ we obtain
\begin{equation}
	z_i\cdot(u_iu_j) 
	= p^{b^2}q^{ab+b^2} u_{i+1}u_j = p^{b^2} r_{ij}r_{j,i+1}q^{2ab}u_{i+1}u_j = 
	  z_i\cdot(r_{ij} u_{j}u_i).
\end{equation}
Similar calculations show $z_i \cdot (u_lu_k) = r_{lk}z_i\cdot(u_ku_l)$, for all $i,k,l$.
\end{proof}
 
 \begin{proposition}\label{prop:invariants}
Let  $H_{n,m} = R\#_\gamma \Sigma_m$ and $A_{a,b}$ as above, for some $a\neq b \:(\mathrm{mod}\: n)$.  If  $\gcd(\det(M_{m,a,b}), n)=1$ and $n$ is even, then $A_{a,b}^{H_{n,m}}$ is contained in the commutative polynomial ring $A_{a,b}^{R}=\K[u_1^n, \ldots, u_m^n]$.
 \end{proposition}
 \newcommand{\N}{\mathbb{N}}
 \begin{proof} Set $A=A_{a,b}$.
 Clearly, $A^H\subseteq A^R$.  Let $u_1, \ldots, u_m$ be the algebra generators of $A$. The monomials $u^\alpha := u_1^{\alpha_1}\cdots u_m^{\alpha_m}$, with $\alpha = (\alpha_1,\ldots, \alpha_m)\in \N^m$, form a $\K$-basis for $A$.   Let $w = \sum a_{\alpha} u^\alpha \in A^R$, with only finitely many non-zero $a_{\alpha}\in \K$. Denote by $\sup(w) = \{ \alpha : a_\alpha\neq 0\}$ the support of $w$. For any $1\leq i \leq m$ we have 
 $$w = x_i \cdot w = \sum a_{\alpha} q^{\alpha_i a + \left(\sum_{j\neq i}\alpha_j\right)b} u^\alpha.$$ Hence 
 $\alpha_i a +  \left(\sum_{j\neq i}\alpha_j\right)b = 0 \:(\mathrm{mod}\: n)$, for any $\alpha \in \sup(w)$ and for any $i$, which means that $\alpha M_{m,a,b} = 0 \:(\mathrm{mod}\: n)$. Since $\gcd(\det(M_{m,a,b}),n)=1$, $M_{m,a,b}$ is invertible modulo $n$ and $\alpha = (0,\ldots, 0) \:(\mathrm{mod}\: n)$, i.e. $n\mid \alpha_i$, for all $i$. Note that if $n$ is even, then $r_{ij}^{n^2} = q^{n^2(b^2-ab)} p^{n^2(b^2-a^2)} = (-1)^{n(b^2-a^2)}=1$. Hence $u_i^n u_j^n=u_j^nu_i^n$ and  $A^R \subseteq \K[u_1^n, \ldots, u_m^n]$. The converse inclusion is clear.
 \end{proof}

The condition $\gcd(\det(M_{m,a,b}), n)=1$ is obviously satisfied, when $a=1$ and $b=0$. Then $\lambda = q^{b^2-ab}=1$ and  $\mu = p^{b^2-a^2}=p^{-1}$. Thus the matrix $(r_{ij})$ satisfies $r_{ji}= p$ for $i<j$.

\begin{theorem}\label{thm:action} Let $n,m\geq 2$, $H=R\#_\gamma \Sigma_m$ and $R=\K\Z_n^{\otimes m}$. Let $q$ be the primitive $n$th root of unity used in the definition of the twist $J$ and assume there exists a square root $p$ of $q$. Then 
$$A = \frac{\K\langle u_1,\ldots, u_m\rangle}{\langle u_ju_i - pu_iu_j : i<j\rangle}$$
 is a left $H$-module algebra with action of the generators $x_i$ and $z_i$ of $H$ given by
$$x_iu_i=qu_i, \quad x_iu_j = u_j \: \mbox{ for $j\neq i$}, \quad z_i u_i = u_{i+1}, \quad z_i u_{i+1}=u_i, \quad z_iu_j=u_j \mbox{ for $j\neq i, i+1$}.$$

Moreover, $A$ is an inner-faithful $R$-module and if $n$ is even then $A^H\subseteq \K[u_1^n, \ldots, u_m^n]$.
\end{theorem}
\begin{proof}
Combine Lemma \ref{lem:inner-faithful}, Theorem \ref{thm:qpoly} and Proposition \ref{prop:invariants}.
\end{proof}

In \cite{Ferraro}, Ferraro et al. showed that Pansera's Hopf algebras $H_{n,2}$ act inner-faithfully on an Artin-Schelter regular algebra $A$, such that its ring of invariants is again Artin-Schelter regular - an analogous property to the Chevalley-Shephard-Todd Theorem for group actions on polynomial algebras. The authors of \cite{Ferraro} termed such Hopf algebras ``reflection Hopf algebra'' for $A$.

 \begin{exa}\label{ex:cyclicHopf}
  Like in Example \ref{exa:cyclic}, consider the Hopf subalgebra $H=R\#_\gamma \langle s\rangle$ of $H_{n,m}$ generated by $R$ and $\theta=z_1z_2\cdots z_{m-1}$, where $s=(12\cdots m)$. Then $H$ has dimension $mn^m$ and is generated as algebra by $x_1,\cdots, x_m$ and $\theta$. By Theorem \ref{thm:action}, $H_{n,m}$ and hence $H$ acts on 
$$A =A_{1,0}= \frac{\K\langle u_1,\ldots, u_m\rangle}{\langle u_ju_i - pu_iu_j : i<j\rangle},$$
where $\theta$ acts cyclically on the generators of $A$, i.e. $\theta \cdot u_j =u_{s(j)}$, for all $j$.
We claim that $A$ is an inner-faithful $H$-module algebra. Suppose  $I$ is a Hopf ideal of $H$  with $I\cdot A = 0$ and let $\omega = \sum_{i=0}^{m-1} a_i \theta^i  \in I$, with $a_i \in R$. Then  
$\omega \cdot u_j = \sum_{i=0}^{m-1} a_i \cdot u_{s(j)} = 0$, for all $1\leq j \leq m$.  Since the spaces $\K u_j$ are stable under the action of $R$, we get $a_i \cdot u_j = 0$, for all $i$ and  $j$. Thus $a_i  \in I\cap R$, for all $i$. As $I$ is also a Hopf ideal of $R$ we conclude by Lemma \ref{lem:inner-faithful}, $I\cap R=0$. Hence $\omega = 0$, which means that $H$ acts inner faithfully on $A$.

 If $n$ is even, then by Theorem  \ref{thm:action}, $A^H\subseteq \K[u_1^n, \ldots, u_m^n]$.  Since $\theta$ acts cyclically on the generators $u_i$, it also acts cyclically on the elements $u_i^n$. Hence $A^H$ consists of the cyclic polynomials in $u_i^n$,  i.e. $A^H = \K[u_1^n, \ldots, u_m^n]^{\langle s \rangle}$. 

 Note that the quantum polynomial algebra $A=A_{1,0}$ is an Artin-Schelter regular ring.
In case $m=2$, $H=H_{n,2}$ and $A^H=\K[u_1^n, u_2^n]^{\langle (12) \rangle} = \K[u_1^n+u_2^n, u_1^nu_2^n]$ is a polynomial algebra, hence Artin-Schelter regular and $H$ is a reflection Hopf algebra for $A$. In case $m\geq 3$, the ring of invariants  $A^H = \K[u_1^n, \ldots, u_m^n]^{\langle s \rangle}$ is not a polynomial ring and  $H$ is not a reflection Hopf algebra for $A$ in the sense of \cite{Ferraro}. 
 \end{exa}

\section*{Acknowledgment}
I  dedicate this paper to the cherished memory of my father. Parts of this research were conducted in 2024 during a visit to the Technical University of Dresden, supported by a DAAD grant. I am especially grateful to Uli Krähmer and his team - Zbigniew Wojciechowski, Ivan Bartulović, and Tony Zorman - for their warm hospitality and for making my stay in Dresden both productive and enjoyable. My thanks goes also to the anonymous referee for her/his valuable comments. Furthermore, I was  partially supported by CMUP, member of LASI, which is financed by national funds through FCT - Fundação para a Ciência e a Tecnologia, I.P., under the projects with reference UIDB/00144/2020 and UIDP/00144/2020.



\begin{bibdiv}
		\begin{biblist}
	\bib{ChenYangWang}{article}{
 author={Chen, Jialei},
 author={Yang, Shilin},
 author={Wang, Dingguo},
 issn={1673-3452},
 issn={1673-3576},
 doi={10.1007/s11464-021-0893-x},
 title={Grothendieck rings of a class of Hopf algebras of Kac-Paljutkin type},
 journal={Frontiers of Mathematics in China},
 volume={16},
 number={1},
 pages={29--47},
 date={2021},
 publisher={Springer, Berlin/Heidelberg; Higher Education Press, Beijing},
}	
		
		\bib{CuadraEtingofWalton}{article}{
   author={Cuadra, Juan},
   author={Etingof, Pavel},
   author={Walton, Chelsea},
   title={Semisimple Hopf actions on Weyl algebras},
   journal={Adv. Math.},
   volume={282},
   date={2015},
   pages={47--55},
   issn={0001-8708},
   doi={10.1016/j.aim.2015.05.014},
}

\bib{Davydov}{article}{
 author={Davydov, A.},
 isbn={978-90-6569-061-6},
 book={
 title={Noncommutative structures in mathematics and physics. Proceedings of the congress, Brussels, Belgium, July 22--26, 2008. Contactforum},
 publisher={Brussel: Koninklijke Vlaamse Academie van Belgie voor Wetenschappen en Kunsten},
 },
 title={Twisted automorphisms of Hopf algebras.},
 pages={103--130},
 date={2010},
}

\bib{EtingofWalton}{article}{
   author={Etingof, Pavel},
   author={Walton, Chelsea},
   title={Semisimple Hopf actions on commutative domains},
   journal={Adv. Math.},
   volume={251},
   date={2014},
   pages={47--61},
   issn={0001-8708},
   doi={10.1016/j.aim.2013.10.008},
}


\bib{Farnsteiner}{article}{
author={Farnsteiner, Rolf},
title={Burnside’s theorem for Hopf algebras.},
 note={Lecture notes, available at
\url{https://www.math.uni-bielefeld.de/ sek/select/rf5.pdf}}
 }


		\bib{Ferraro}{article}{
   author={Ferraro, Luigi},
   author={Kirkman, Ellen},
   author={Moore, W. Frank},
   author={Won, Robert},
   title={Three infinite families of reflection Hopf algebras},
   journal={J. Pure Appl. Algebra},
   volume={224},
   date={2020},
   number={8},
   pages={106315, 34},
   issn={0022-4049},
   doi={10.1016/j.jpaa.2020.106315},
}

\bib{KacPaljutkin}{article}{
   author={Kac, G. I.},
   author={Paljutkin, V. G.},
   title={Finite ring groups},
   language={Russian},
   journal={Trudy Moskov. Mat. Ob\v s\v c.},
   volume={15},
   date={1966},
   pages={224--261},
   issn={0134-8663}
}

\bib{Sebastian}{misc}{
 author={Sebastian Halbig},
 author={Christian Lomp},
 review={arXiv:2511.11047},
 state={submitted},
 title={Irreducible representations of generalised Kac-Paljutkin Hopf algebras},
 date={2025},
}

\bib{LompPansera}{article}{
   author={Lomp, Christian},
   author={Pansera, Deividi},
   title={A note on a paper by Cuadra, Etingof and Walton},
   journal={Comm. Algebra},
   volume={45},
   date={2017},
   number={8},
   pages={3402--3409},
   issn={0092-7872},
   doi={10.1080/00927872.2016.1236933},
}


\bib{Masuoka}{article}{
 author={Masuoka, Akira},
 issn={0021-2172},
 issn={1565-8511},
 doi={10.1007/BF02762089},
 title={Semisimple Hopf algebras of dimension {{\(6\)}}, {{\(8\)}}},
 journal={Israel Journal of Mathematics},
 volume={92},
 number={1-3},
 pages={361--373},
 date={1995},
 publisher={Springer, Berlin/Heidelberg; Hebrew University Magnes Press, Jerusalem},
}

\bib{Masuoka1997}{article}{
 author={Masuoka, Akira},
 issn={0021-8693},
 issn={1090-266X},
 doi={10.1006/jabr.1996.6863},
 review={Zbl 0877.16019},
 title={Calculations of some groups of Hopf algebra extensions},
 journal={Journal of Algebra},
 volume={191},
 number={2},
 pages={568--588, art. no. ja966863},
 date={1997},
 publisher={Elsevier (Academic Press), San Diego, CA},
}

\bib{Masuoka2002}{article}{
 author={Masuoka, Akira},
 isbn={0-521-81512-6},
 book={
 title={New directions in Hopf algebras},
 publisher={Cambridge: Cambridge University Press},
 },
 review={Zbl 1011.16024},
 title={Hopf algebra extensions and cohomology},
 pages={167--209},
 date={2002},
 eprint={www.msri.org/communications/books/Book43/},
}


\bib{PanseraPhD}{book}{
   author={Pansera, Deividi Ricardo},
   title={On Semisimple Hopf Actions},
   note={Thesis (Ph.D.)},
   publisher={Universidade do Porto},
   date={2017},
   pages={103}
}
\bib{Pansera}{article}{
   author={Pansera, Deividi},
   title={A class of semisimple Hopf algebras acting on quantum polynomial
   algebras},
   conference={
      title={Rings, modules and codes},
   },
   book={
      series={Contemp. Math.},
      volume={727},
      publisher={Amer. Math. Soc., [Providence], RI},
   },
   isbn={978-1-4704-4104-3},
   date={[2019] \copyright 2019},
   pages={303--316},
   doi={10.1090/conm/727/14643},
}
\bib{Passman}{article}{
   author={Passman, D. S.},
   author={Quinn, Declan},
   title={Burnside's theorem for Hopf algebras},
   journal={Proc. Amer. Math. Soc.},
   volume={123},
   date={1995},
   number={2},
   pages={327--333},
   issn={0002-9939},
   doi={10.2307/2160884},
}

\end{biblist}
\end{bibdiv}
\end{document}